\documentclass[11pt]{article}
\usepackage{amsmath,amssymb,amsthm,color,graphicx}
\usepackage[latin1]{inputenc}
\usepackage{amsfonts}
\usepackage{wasysym}
\newcommand{\R}{\mathbb R}
\newcommand{\eps}{\epsilon}

\newcommand{\vp}{\vec{p}}
\newcommand{\vq}{\vec{q}}
\newcommand{\vm}{\vec{m}}
\newcommand{\Jm}{J^{-}}

\newcommand{\vx}{\vec{x}}

\newcommand{\vpsi}{\vec{\psi}}
\newcommand{\vM}{\vec{M}}
\newcommand{\vA}{\vec{A}}

\newcommand{\vt}{\vec{t}}

\newcommand{\bA}{\bar{A}}

\newtheorem{theorem}{Theorem}
\newtheorem{lemma}{Lemma}[section]
\newtheorem{prop}{Proposition}[section]
\newtheorem{cor}{Corollary}[section]
\newtheorem{remark}{Remark}[section]
\newtheorem{defi}{Definition}[section]

\newtheorem{example}{Example}[section]
\newtheorem{acknowledgment*}{Acknowledgment}
\newtheorem{assumption}{Assumption}[section]
\newcommand{\be}{\begin{equation}}
\newcommand{\ee}{\end{equation}}

\begin{document}
\Large\begin{center} {\bf On optimal partitions, individual values and cooperative  games: Does  a wiser agent always  produce a higher value? }\end{center}
\normalsize
\begin{center} Gershon Wolansky\footnote{Department of Mathematics, Technion 32000, Israel}\end{center}
\begin{abstract}
We consider an optimal partition of  resources (e.g. consumers) between several agents, given utility functions ("wisdoms") for the agents and their capacities. This problem is a  variant of  optimal transport (Monge-Kantorovich) between two measure spaces   where one of the measures  is discrete (capacities) and the costs of transport are the wisdoms of the agents. We concentrate on the individual  value  for each agent under optimal partition  and show that, counter-intuitively, this value may decrease if the agent's  wisdom is increased.   Sufficient
 and necessary conditions    for the  monotonicity with respect to the wisdom functions of the individual values   will be given, independently of the other agents. The sharpness of these conditions is also discussed.
\par
Motivated by the above  we define a cooperative  game based on optimal partition and investigate conditions for stability of the grand coalition.
\end{abstract}

\section{Introduction}\label{OP}
Let $X$ be a set of  "consumers".  Let $I:= \{1, \ldots N\}$ be a set of "agents".
 For  each  agent  $i\in I$ we associate a  real  valued function $\psi_i$ on $X$. This symbolizes the agent's  "wisdom":    $\psi_i(x)$ is the profit that agent  $i$ can make for a consumer $x$, if the latter  uses her service. We denote the "wisdom vector" $\vpsi:= (\psi_1, \ldots \psi_N):X\rightarrow \R^N_+$.
 \par
 Let $\mu$ be  the distribution of consumers in $X$. In particular $\mu(X)$ is the number of consumers.

  \par
  Each agent $i$ has a limited capacity  $M_i>0$.  This symbolizes the total number of consumers she  can serve. In other words,
an agent $i$ can serve a measurable set $A_i\subset X$ of consumers for which $\mu(A_i)\leq M_i$. Another assumption is that each consumer can hire {\em at most} one agent,  meaning  $\mu(A_i\cap A_j)=0$ if $i\not= j$.
\par
The profit made by an agent $i$ for her consumers  $A_i$
 is just   $$P_{\psi_i}(A_i):=\int_{A_i} \psi_id\mu \ . $$
Let $\vM:= (M_1, \ldots M_N)\in\R^N_+$, and
${\cal P}_{\vM}$ the set of all essentially disjoint measurable partitions  of $X$, that is
 $$ {\cal P}_{\vM}:= \{ \vA:=(A_1, \ldots A_N); \  A_i\subset X, \ \mu(A_i)\leq M_i  \ \ \text{and} \ \  \mu(A_i\cap A_j)=0 \ \ \text{if} \ \ i\not= j\in I \ . \}$$
 We distinguish between three  cases:
 \begin{description}
 \item{Over Saturation (OS)}: $\sum_{i\in I} M_i> \mu(X)$.
 \item{Saturation (S)}:  $\sum_{i\in I} M_i=\mu(X)$.
 \item {Under-Saturation (US)}: $\sum_{i\in I} M_i<\mu(X)$.
 \end{description}

\paragraph{First paradigm:}  The big brother
  (or "invisible hand", or macro "Keynesian distributer") splits the consumers $X$ between the agents in order to maximize their  total profit, taking account the capacity constraints. Let
\be\label{defSigma}\Sigma_{\vpsi}(\vM):= \sup_{\vA\in {\cal P}_{\vM}} \sum_{i\in I} P_{\psi_i}(A_i) \  \ee
be the maximal total profit of the consumers.
\par
 Suppose a maximizing   partition  $\vec{\bar{A}}(\vM,\vpsi):= \left(\bA_1(\vM,\vpsi), \ldots \bA_N(\vM,\vpsi)\right)$ exists. The profit made by  each agent $i$   for herself  and her consumers under this optimal partition  is her {\it Individual Value} (i.v):
 \be\label{barP}\bar{P}_i(\vM, \vpsi):= \int_{\bar{A}_i(\vM,\vpsi)}\psi_id\mu \ . \ee
\begin{remark}
A maximizing  partition can be interpreted as a {\it Pareto efficient} plan.
\end{remark}

\paragraph{Second paradigm:}
Assume there is no big brother. The market is free, and each consumer may choose his favorite agent to maximize his own  utility.   Each agent determines the price she collects for consulting a consumer. Let $p_i\in\R$ the price requested by agent $i$, $\vp:= (p_1, \ldots p_N)$. The utility of a consumer $x$ choosing the agent $i$ is, therefore, $\psi_i(x)-p_i$, if it is positive. If $\psi_i(x)-p_i\leq 0$ then the consumer will avoid the agent $i$, so he pays nothing and gets nothing in return. The net income of consumer $x$ choosing agent $i$ is, therefore, $[\psi_i(x)-p_i]_+:= \max(\psi_i(x)-p_i, 0)$.
\par
Note that we do not assume $p_i\geq 0$. In fact, we also takes into account  "negative prices"  (bonus, or bribe).
\par
The set of consumers who give up counseling by any of the agents is
\be\label{A+0} A^+_0(\vp):=\{x\in X;  \psi_i(x)-p_i\leq 0 \ \ \text{for any} \ i\in I\} \  \ee
while the  set of consumers who prefer agent  $i$ is, then
\be\label{A+i} A^+_i(\vp):= \{x\in X; \psi_i(x)-p_i\geq  \psi_j(x)-p_j \ \ \forall j\in I\}-A^+_0(\vp) \ .  \  \ee
Assumption \ref{funda} formulated in Section \ref{firstsec} below guarantees that $\{A^+_i(\vp)\}$ are essentially mutually disjoint  for any $\vp\in\R^N$, i.e. $\mu(A_i^+(\vp)\cap A^+_j(\vp))=0$ for any $i\not=j\in I\cup\{0\}$.
\begin{defi}
The vector  $\vp_0:= \{ p_{0,1} , \ldots p_{0,N}\}\in \R^N$ is  an  {\em equilibrium price} vector  with respect to $\vM$ if the set  of  the consumers who choose  agent $i$ meets her capacity constraint, namely
$$ \mu(A^+_i(\vp_0))\leq M_i \  \ \forall i\in I\cup \{0\}  $$
where $M_0:=[\mu(X)-\sum_{i\in I}M_i]_+$. In particular, $\vec{A}^+(\vp_0):=(A_1^+(\vp_0), \ldots A_N^+(\vp_0))\in {\cal P}_{\vM}$.
\end{defi}
\begin{remark}\label{full}
Note that $M_0=0$ in the S, OS cases, so an equilibrium price vector must satisfy $\mu(A_0(\vp_0))=0$, and  $\sum_{i\in I}\mu(A_i(\vp_0))=\mu(X)$.
\end{remark}
\begin{remark}
An equilibrium price $\vp_0$ can be interpreted as a Walrasian equilibrium.
\end{remark}
\paragraph{Big brother meets the Free market}: \
 Theorem \ref{old} in Section \ref{firstsec} below  claims that there exists an equilibrium price vector $\vp_0$ which realizes the optimal partition of the big brother.
  The corresponding partition $\vec{A^+}(\vp_0)\in \cal P_{\vM}$ is unique (even though the price vector $\vp_0$ is not necessarily so).
  \par
  Theorem \ref{old} is strongly related to the celebrated theory of optimal mass transport (Monge Kantorovich). It can also be considered as a special case of the second Welfare Theorem.
  \par
  The modern theory of optimal transport (OT)
 is a source of many research papers in mathematical analysis \cite{A1}, \cite{AGS}, \cite{Br}, probability \cite{R}, \cite{Rach}, \cite{cl}, \cite{AS}, geometry \cite{Vil1}, \cite{Vil2}, \cite{Gan}, \cite{Mc}, PDE \cite{Cap}, \cite{Ev},\cite{MT}, \cite{Egan} and many other fields in (and outside) pure Mathematics \cite{Ru1}, \cite{Ru2}, \cite{RYT}, \cite{GO}, \cite{carl}, \cite{CL}.
  It generalizes  the discrete {\it assignment problem} \cite{PD} to general (infinite dimensional) measure spaces.
\par
The object of OT is very intuitive \cite{mon}. We are given two measure spaces: An atomless probability  measure $\mu$ on  $X$, and (not necessarily atomless) probability measure $\nu$ on  a measure space $Y$. In addition there is  a measurable value function $c:X\times Y\rightarrow \R$.  An optimal solution is  an admissible, measurable   mapping $T:X\rightarrow Y$  which maximizes $T \mapsto\int_Xc(x, T(x))d\mu(x)$ subject to the constraint  $T_\#\mu=\nu$, namely $\mu(T^{-1}A)=\nu(A)$ for any measurable set $ A\subset Y$.
\par

Optimal partition (OP) is a  particular aspect of OT theory \cite{W}. It deals with the case where   one of the probability spaces, say $Y$, is a finite discrete one $Y\equiv I:=\{1, \ldots N\}$. In that case the value function $c$ is reduced to $N$ functions on $X$, named $\psi_i(x):=c(x, \{i\})$ ("wisdoms")   and $\nu(\{i\}):= M_i\geq 0$ ("capacities") satisfy $\sum_{i\in I}M_i=\mu(X)$.
 An admissible mapping $T:X\rightarrow I$ is a  partition $A_i:= T^{-1}(\{i\})\subset X$ into measurable subsets of $X$  verifying $\mu( A_i\cap A_j)=0$ for $i\not= j$ and $\mu\left( A_i\right)=M_i$.  The  Monge problem is reduced, in that case, to finding such a partition  which  optimizes  $\sum_{i\in I}\int_{A_i}\psi_id\mu$ under these constraints.
 \par
 The result indicated in Theorem \ref{old}, Section \ref{firstsec}, if restricted to  the saturated case $\sum_1^NM_i=\mu(X)$,   can be seen as a special case of the general Kantorovich duality Theorem \cite{Vil1}. The uniqueness result of Theorem \ref{old} follows from  Assumption \ref{funda} which is a generalization of the {\it twist condition} (c.f \cite{cham}).  The under-saturated case can also be included under this umbrella (c.f. Remark \ref{rem52}). The only new feature (perhaps) of this Theorem is in the over-saturated case.
\par
 The existence of a unique partition guaranteed in Theorem \ref{old}  enables us to define {\it individual values} generated by an agent (\ref{barP})
 (c.f. also Definition \ref{uni} in  Section \ref{indipro}). We interpret this individual value as the sum of the {\it profit of the agent and her consumers}. In Section \ref{differentinter} we will introduce alternative definitions of individual values, which are outside the scope of this paper.
 \par
  Next, we address the question of monotonicity of the individual values with respect to the wisdom of the agents. Intuitively, it seems that an increase of the wisdom of an agent  {\it without changing the wisdoms of the other agents, nor the capacity of any of the agents} will improve the competitiveness   of this agent and contribute to her individual value. It turns out  that this is not always the case. We obtain sharp estimates on the individual value under such assumptions  in Theorems \ref{new0}-\ref{new1} in Section \ref{indipro}.
 \par
The discussion on individual values leads naturally to study of {\it coalitions} and {\it cartels}. Suppose, under the big brother paradigm, that the agents in a  subset  $J\subset I$ decide to join forces and make a coalition. They offer any potential consumer the best agent (for him) in that coalition, and unite their capacities. This leads to a new partition problem where the agents in $\{i\in J\}$ are replaced by a single "super-agent" $J$ whose wisdom is $\psi_J(x):=\max_{i\in J}\psi_i(x)$, and whose capacity is $M_J:=\sum_{i\in J} M_i$. The big brother now chooses  the optimal partition between this super agent $J$ and the other agents (or other super-agents due to other coalitions). In particular, if a {\it grand} coalition $J=I$ is formed, then the "individual profit" of this coalition is just $P_I\equiv \int \psi_Id\mu:= \int \max_{i\in I} \psi_i(x)d\mu(x)$.
 \par
  Under the free market paradigm, the agents in a coalition $J$  decide to unify the price (i.e., create a cartel)  for  their services, so $p_i\equiv p_J$ for any $i\in J$. The free market  now selects the optimal price vector, under that constraint. It turns out that the partition due to this optimal price is the same as the partition due to  the big brother paradigm (Proposition \ref{propcoal}).
 \par
 If such a coalition/cartel  is declared, the natural question is how to distribute the individual value  of the coalition between its members?  In particular, if a grand coalition is formed, an agent $i$  will be happy  if her share in the profit  is not smaller than her share without joining the coalition, or her share in a different  (smaller) coalition $J\supset \{ i\}$.
 \par
  A cooperative game is defined by a  real valued function  $\nu$ acting on the subsets of the set of agents: $\nu:2^I\rightarrow \R$. For any $J\subset I$, $\nu(J)$ is  the reward for the  coalition $J$. In section \ref{ColPar} we define such a coalition game where $\nu(J)$ is the individual  value of a coalition $J$ in the optimal partition between two super-agents: $J$ and its complement $J^-:=I-J$. After reviewing some concepts from cooperative game theory (section \ref{gamerev}) we discuss the stability of the grand coalition in the special cases where all wisdoms are multiples of a single one, i.e.
  $\psi_j(x)\equiv \lambda_j\psi(x)$ where $\lambda_j>0$ are constants.
  \par
  In Section \ref{duality} we introduce the  tools for the proofs of Theorems \ref{old}-\ref{new1}.

 \section{Existence and uniqueness of optimal partitions}\label{firstsec}
\begin{assumption}\label{funda} \
 \begin{itemize}
\item  $X$ is  a compact  Hausdorff space,  $\mu$ a  probability Borel measure on $X$ and $\psi_i$, $i\in I$ are continuous, real valued on $X$ which are positive $\mu$ a.e, that is,  \be\label{nl} \sum_{i\in I}\mu(x; \psi_i(x)\leq 0)=0 \ . \ee
  \item   For any $r\in\R$ and any $i\not=j\in I$
    \be\label{post} \mu(x; \psi_i(x)-\psi_j(x)=r)=0\   \ee
\item In the US case we further assume
\be\label{post0}\mu(x; \psi_i(x)=r)=0 \   \ , \ \forall i\in I \  \ .  \ee
\end{itemize}
\end{assumption}

\begin{theorem} \ \label{old}
Under Assumption \ref{funda}
\begin{description}
\item {US:} In the under-saturation case ($\sum_{i\in I}M_i<\mu(X)$), let
 \be\label{Xi0}\Xi^0_{\vpsi}(\vp):=\int_X\max_{i\in I}[\psi_i(x)-p_i]_+d\mu(x) \ . \ee
  Then there exists a unique equilibrium price  vector $\vp_0$ which is a minimizer of  $\Xi^0_{\vpsi}(\vp)+\vp\cdot\vM$ on $\R^N$. Moreover, the associated partition $A_i^+(\vp_0)$ (\ref{A+i}) is the {\em unique} optimal partition which maximizes (\ref{defSigma})  and we get
\be\label{s000} \Sigma_{\vpsi}(\vM)= \min_{\vp\in\R^N} \Xi^0_{\vpsi}(\vp)+\vp\cdot\vM \equiv \Xi^0_{\vpsi}(\vp_0)+\vp_0\cdot\vM \ . \ee
\item {S:} In the saturation case ($\sum_{i\in I}M_i=\mu(X)$),
 there exists an equilibrium price  vector $\vp_0$   which satisfies
  (\ref{s000}) as well. This vector is unique only up to an additive shift $p_{0,i}\rightarrow p_{0,i}+\gamma$ where $\gamma\leq 0$.
   The associated partition $A_i^+(\vp_0)$ (\ref{A+i}) is the {\em unique} optimal partition which maximizes (\ref{defSigma}).

\item {OS:} In the over-saturation case ($\sum_{i\in I}M_i>\mu(X)$)  the existence and uniqueness are as in the saturated case, and the actual capacity $\hat{M}_i\leq M_i$ of each agent is {\em uniquely} determined by the maximizer of
\be\label{maximum} \vm \mapsto \Sigma_{\vpsi}(\vm):=  \min_{\vp\in \R^N} \Xi^0_{\vpsi}(\vp)+\vp\cdot\vm, \ \ \ \ \vm\leq \vM \ . \ee
\end{description}
In all the above, the unique optimal partition is given by (\ref{A+0},\ref{A+i}) where $A_0^+=\emptyset$ in the S, OS cases.

\end{theorem}

%
%
The proof of Theorem \ref{old} is given in Section \ref{funda}.

\begin{remark}\label{sineq}
In the saturation (S) or under saturation  (US) cases, the optimal partition $\bar{A}_i$ must meet the limit of the capacity constraint, $\mu(\bar{A}_i)=M_i$.  Indeed, if a strong inequality  $\hat{M}_i:=\mu(\bar{A}_i)<M_i$ holds for some $i\in I$, then the non-consumers set $\bar{A}_0:= X-\cup_{i\in I} \bar{A}_i$ satisfies  $\mu(\bar{A}_0)>0$ and the big brother could increase the total profit $\Sigma_{\vpsi}(\vM)$ by assigning the consumers from  $\bar{A}_0$ to  agent $i$, c.f.  (\ref{nl}).
\end{remark}
\begin{remark} $A^+_0(\vp)$  (\ref{A+0}) is a closed set since $\psi_i$ are continuous (Assumption \ref{funda}). The sets $A_i^+(\vp)$ are essentially disjoint by definition and assumption (\ref{post}), i.e $\mu(A_i^+(\vp)\cap A^+_j(\vp))=0$ for $i\not=j\in I\cup 0$. Moreover, $X=\cup_{i\in I\cup\{0\}} A_i^+(\vp)$.
\end{remark}

\begin{remark}\label{remhadash}
We may determine a {\em unique} vector price in the saturation (as well as the over-saturation case) by assigning the reasonable condition that the agents set their prices as high as they can. This is the maximal possible price for which $\mu(A^+_0)=0$:
$$ \vp_{\bar{\gamma}}:= \vp_0+\bar{\gamma}\vec{1}$$
 where
 $$  \vec{1}:=(1, \ldots 1)\in \R^N, \ \  \bar{\gamma}:= \inf\{ \gamma\in \R \ ; \ A_0^+(\vp_0+\gamma\vec{1}) =\emptyset \ \} \ .  $$

\end{remark}
\begin{example}\label{exAlambda}
Assume  $N:=|I|=2$ and the saturated case  $M_1+M_2=\mu(X)$.
Let
\be\label{com1} A_i^\lambda:= \left\{x\in X; \ \ \psi_i(x)\geq \psi_{j}(x)+\lambda \right\} \  \ee
where $i\not=j\in\{1,2\}$. Indeed, $A_0^+(\vp_0)=\emptyset$ in the saturated case, so $A_i^\lambda=A_i^+(\vp_0)$ where $\lambda=p_i-p_j$ in that case ($|I|=2$).   By Theorem \ref{old}, the optimal share of agent $i$ is determined by the condition $\mu(A_i^\lambda)=M_i$. Thus,
  $\bar{A}_i=A_i^{\bar{\lambda}(M_i)}$  is the optimal share of agent $i$, where
\be\label{com2} \bar{\lambda}(M):= \sup\{\lambda; \mu(A_i^\lambda) >M\} = \inf\{\lambda; \mu(A_i^\lambda) <M\} \ . \ee
In particular, if $\{x_i\}$ the set of maximizers of $\psi_i-\psi_j$ in $X$ then (recall (\ref{barP}))
$$ \lim_{M_i\rightarrow 0} M_i^{-1}\bar{P}_i(\vM,\vpsi)\equiv\lim_{M_i\rightarrow 0} \mu^{-1}(A_i^{\lambda(\vM_i)})\int_{A_i^{\lambda(\vM_i)}}\psi_id\mu\in Conv\{\psi_i(x_i)\} \ ,  $$
where $Conv\{\cdot\}$ stands for the convex hull, and the limit is uniquely determined  if $\{x_i\}$ is a singleton.

C.f Fig \ref{fig1}.

\end{example}
\begin{figure}[h]
  \centering
   \includegraphics[width=15cm]{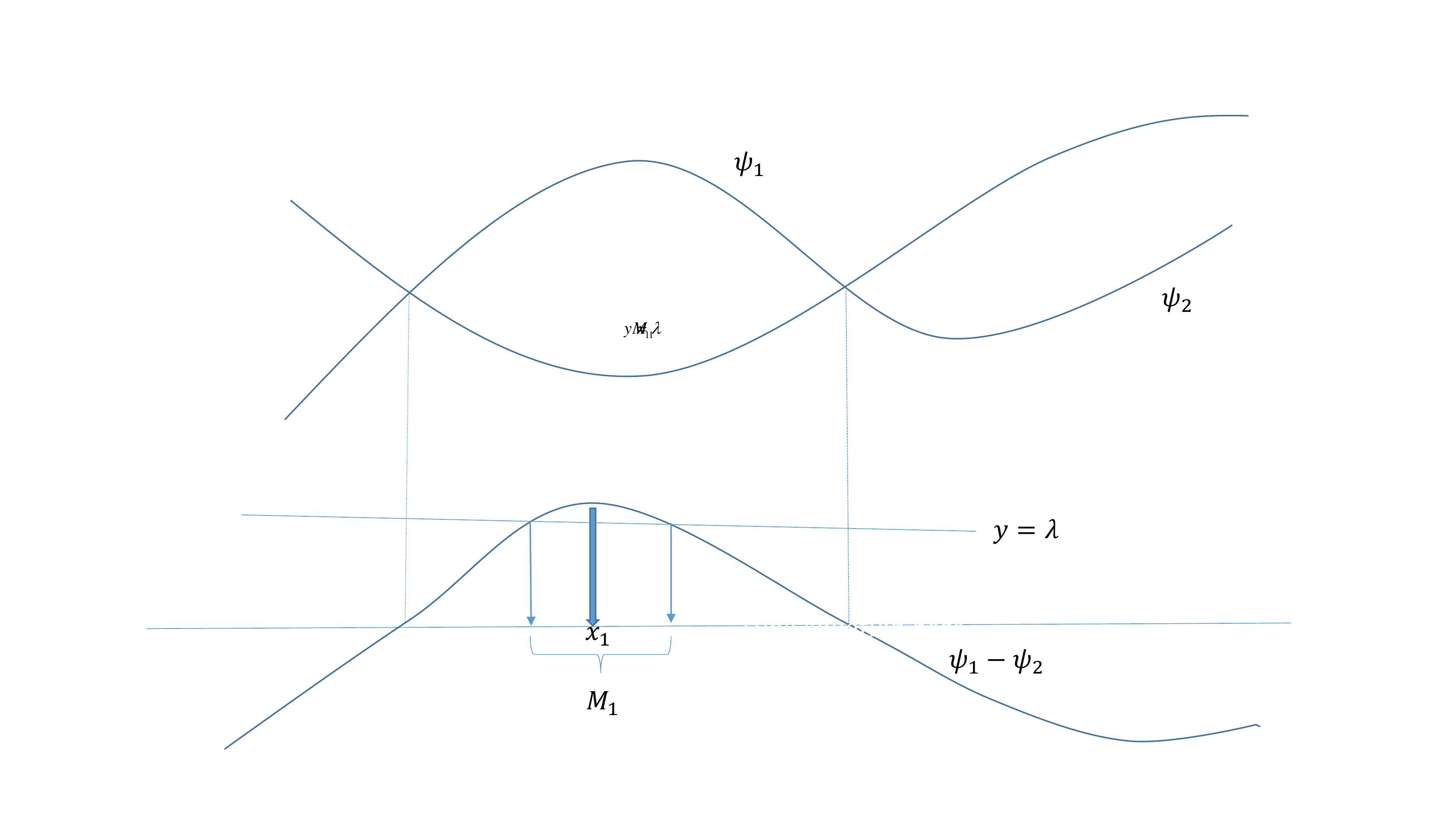}
       \caption{The optimal share $A_1$ of capacity $M_1$ is determined by the level line of $\psi_1-\psi_2\geq\lambda$, containing the maximal point $x_1$ of this function.}
       \label{fig1}
\end{figure}

\begin{example}\label{exlp}
 Suppose $\psi$  is a non-negative, continuous function on $X$ verifying $\mu(x; \psi(x)=r)=0$ for any $r\geq 0$. Let $\lambda_N>\lambda_{N-1}>\ldots \lambda_1>0$ be constants. We assume $\psi_i:=\lambda_i\psi$ verifies assumption \ref{funda}.

From (\ref{A+0}, \ref{A+i}) we obtain
$$ A_0^+(\vp)\equiv \{ x; \psi(x)<\min_{i\in I}\lambda_i^{-1}p_i\} \ ; \ \ A_i^+(\vp)\equiv \left\{ x; \min_{j>i}\frac{p_j-p_i}{\lambda_j-\lambda_i} > \psi(x) > \max_{j<i}\frac{p_j-p_i}{\lambda_j-\lambda_i}\right\}-A_0^+(\vp) \ . $$
In particular, the partition $A_i^+(\vp)$ are consists  of unions of level sets of the function $\psi$.
\par
At optimal partition we observe that the share of the "top agent" $N$ is just the level set
  $$ A_N:=A^{\bar{\lambda}_N}\equiv  \left\{x\in X; \ \ \psi(x)\geq \bar{\lambda}_N \right\} \  $$
where
$$ \bar{\lambda}_N:= \sup\{\lambda; \mu(A^\lambda) >M_N\} = \inf\{\lambda; \mu(A^\lambda) <M_N\} \  $$
(compare with (\ref{com1}, \ref{com2})). By induction we can proceed to find
 \be\label{lambdabar}A_{j}:=A^{\bar{\lambda}_{j+1},\bar{\lambda}_{{j}}}\equiv \left\{x\in X; \ \ \bar{\lambda}_{j+1}\geq \psi(x)\geq \bar{\lambda}_{j} \right\} \  ,  \ee
where $\bar{\lambda}_{j+1}$ is known  by the induction step and
$$ \bar{\lambda}_{j}:=  \inf\{\lambda\geq 0; \mu(A^{\bar{\lambda}_{j+1},\lambda}) <M_{j}\} \ . $$
An equivalent  representation of the optimal partition is obtained as follows:
 Let $A_s:= \{x\in X; \psi(x)\leq s\}$, $f(s):=\mu(A_s)$, $F(t):=\int_0^t f(s)ds$, $F^*(m):=\sup_{t\in\R}[mt-F(t)]$. Note that $F$ is a convex function and $F^*$ is its Legendre transform, which is convex as well  and is defined on the interval $[0,\mu(X)]$.  Moreover, by duality
 $$ F^*(\mu(A_s))\equiv\int_{A_s}\psi d\mu  .  $$
 \par
 Given $\vM\in \R^N_+$, let ${\cal M}_{N+1}:=\mu(X)$, ${\cal M}_{j}:= [{\cal M}_{j+1}-M_j]_+$ for $j=N, N-1, \ldots 1$. Then $\bar{\lambda}_i=F^*({\cal M}_{i})$, so (\ref{lambdabar}) is written as
 $$A_i:=\{x\in X; F^*({\cal M}_{i})<\psi(x)\leq  F^*({\cal M}_{i+1})\} \ \ \ \ ; \  \ N\geq i\geq 1 \ . $$
  Note that the optimal partition does not depend on $\vec{\lambda}:=(\lambda_1, \ldots, \lambda_N)$, as there is no relation between $\lambda_i$ and $\bar{\lambda}_i$.  However, the total profit $\Sigma_{\vpsi}(\vM)$ does, of course, depend on $\vec{\lambda}$:
  \be\label{unidi}\Sigma_{\vpsi}(\vM)=\sum_{i=1}^N \lambda_i\int_{A_i}\psi d\mu\equiv \sum_{i=1}^N \lambda_i(F^*({\cal M}_{i+1})-F^*\left({\cal M}_{i})\right) \ . \ee
  C.f Fig.  \ref{1p6}.
 \end{example}
\begin{figure}[h]
  \centering
    \includegraphics[width=15cm]{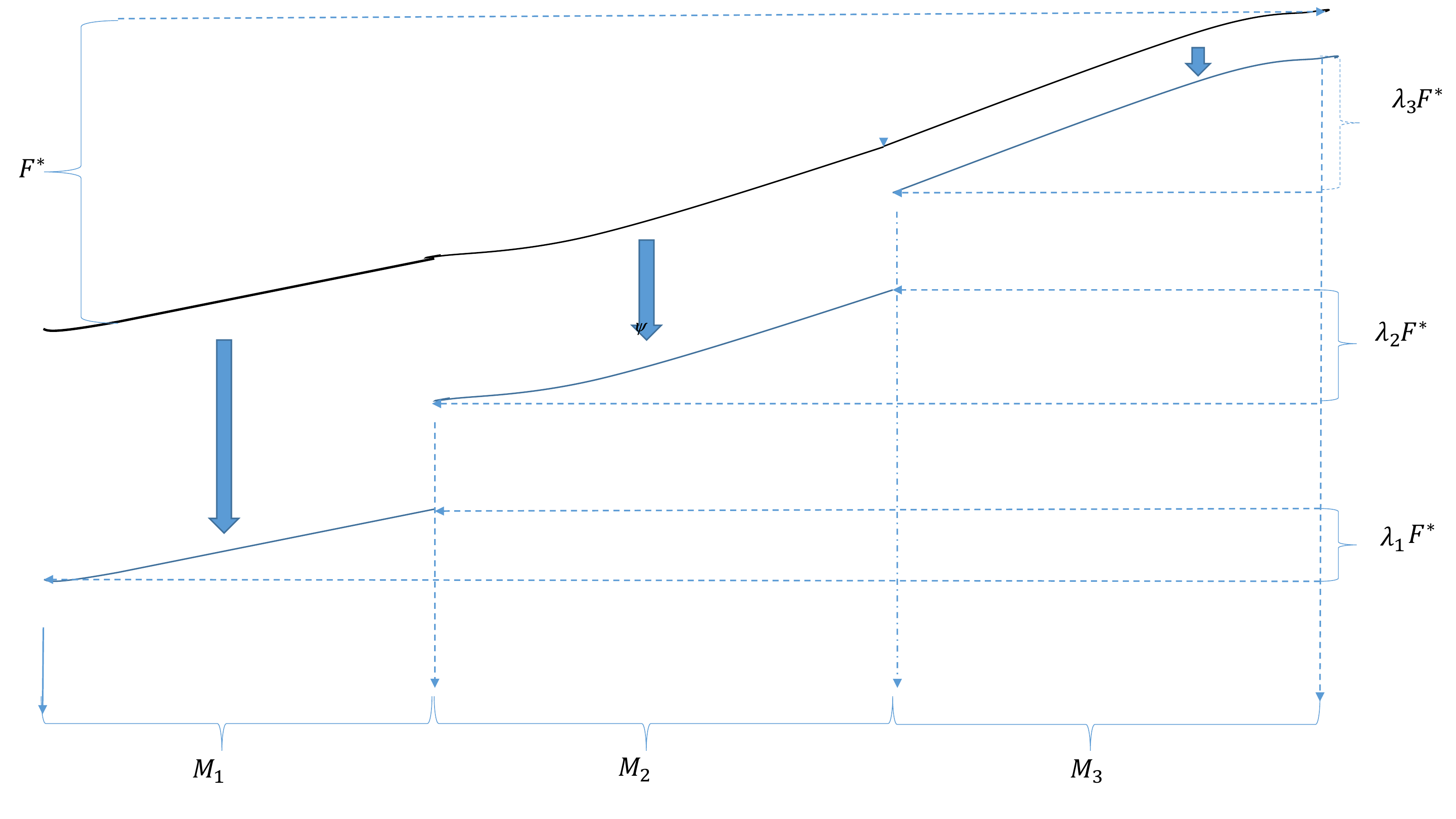}
       \caption{Representation of  $s\mapsto F^*(s))$ corresponding to $\psi$, and optimal partition for $\psi_i=\lambda_i\psi$,$i=1,2,3$ subjected to capacities $\vM=(M_1, M_2, M_3)$ at saturation.  Here $\lambda_3<1$.}
       \label{1p6}
\end{figure}

\section{Individual values}\label{indipro}
At the equilibrium price $\vp_0$, the profit of agent $i$ is just $p_{0,i}\mu(A_i^+(\vp_0))$. The profit gained by the set of her consumers $A_i^+(\vp_0)$ is
$$\int_{A^+_i(\vp_0)} \left(\psi_i(x)-p_{i,0}\right) d\mu \ . $$
The {\it sum} of the profits is
\be\label{surp}\bar{P}_i:=\int_{A_i^+(\vp_0)}\psi_id\mu \ , \ee
which, by Theorem \ref{old}, equals the consumer's part at the optimal partition $\bA_i(\vM,\vpsi)$ in the "Big Brother" paradigm given by (\ref{barP}). We denote $\bar{P}_i(\vM,\vpsi)$ as the {\em individual value} (i.v) of agent $i$, occasionally omitting its dependence on $\vM, \vpsi$.

\begin{assumption}\label{uni} Let $(\vpsi, \vM)$ verify Assumption \ref{funda}. Let $\tilde{\vpsi}$ given by $\tilde{\psi_i}=\psi_i$ for $i\not= 1$.
In other words,  the wisdom $\psi_1$ of agent 1 is replaced  by $\tilde{\psi}_1$ {\em while
keeping $\psi_i$, $2\leq i\leq N$ and $\vM$ unchanged}. The
  new system $(\tilde{\vpsi}, \vM)$ verifies Assumption \ref{funda} as well. We denote    $\tilde{P}_1$   the i.v of agent 1 after this change takes place.
\end{assumption}
\paragraph{Do wiser agents always get higher values?}

Assuming a system of two agents in saturation and suppose that, after some education and training,  agent $1$ improves her wisdom, so, in the notation of Assumption \ref{uni},  $\tilde{\psi}_1\geq \psi_1$ on $X$.
 We expect that the i.v of  the first agent $\bar{P}_1$ will increase under this change, so $\tilde{P}_1\geq \bar{P}_1$.  Is it so, indeed?
 \par
 Well, not necessarily! Suppose $M_1<<\mu(X)$ and let $x_1$ be a unique maximizer of $\psi_1-\psi_2$. According to  Example \ref{exAlambda}, $\bar{P}_1 \approx M_1\psi_1(x_1)$. Let  $\tilde{x}_1$ be a unique  maximizer of $\tilde{\psi}_1-\psi_2$. Now,
 $\tilde{P}_1\approx M_1\tilde{\psi}_1(\tilde{x}_1)$.  But it may happen  that  $\tilde{\psi}_1(\tilde{x}_1)< \psi_1(x_1)$, even though $\tilde{\psi}_1>\psi_1$ everywhere!
 \par
 Definitely, there are cases for which  an increase in the wisdom of a given agent will increase her i.v, independently of her own capacity, as well as  the wisdoms and capacity  of the other agents. In particular,  we can think about two cases where the above argument fails:
 \begin{description}
 \item{Case 1:} \ If $\tilde{\psi}_1=\psi_1+\lambda$, where $\lambda>0$ is a constant.
 \item{Case 2:} \ if  $\tilde{\psi}_1=\beta \psi_1$, where $\beta\geq 2$ is a constant.
 \end{description}
 In the first case the "gaps" $\psi_1-\psi_2$ and $\tilde{\psi}_1-\psi_2$ preserve their order, so if $x_1$ is a maximizer of the first, it is also a maximizer of the second. In particular  the optimal partition $\bar{A}_1, \bar{A}_2$ is unchanged, and we can even predict that $\tilde{P}_1=\bar{P}_1+\lambda M_1>\bar{P}_1$ (c.f Theorem \ref{new0} below).
 \par
 In the second case the order of gaps may change. It is certainly possible that $\tilde{\psi}_1(\tilde{x}_1)-\psi_2(\tilde{x}_1)> \tilde{\psi}_1(x_1)-\psi_2(x_1)$ (where $x_1$, $\tilde{x}_1$ as above), but, if this is the case, an elementary calculation yields $\tilde{\psi}_1(\tilde{x}_1)>
 \psi_1(x_1)$, so the above argument fails. Indeed, if we assume both $\beta\psi_1(\tilde{x}_1)-\psi_2(\tilde{x}_1)> \beta\psi_1(x_1)-\psi_2(x_1)$ and $\beta\psi_1(\tilde{x}_1)<
 \psi_1(x_1)$, then $\psi_1(x_1)-\psi_2(x_1)< -\psi_2(\tilde{x}_1)<\psi_1(\tilde{x}_1)-\psi_2(\tilde{x}_1)$ (since $\psi_1\geq 0$ and $\beta\geq 2$), so  $x_1$ cannot be the maximizer of $\psi_1-\psi_2$ as assumed.
 \par
 In fact, more can be said:
 \begin{theorem}\label{new0}
 Under Assumption  \ref{uni},
 \begin{description}
 \item{i)} If $\tilde{\psi}_1=\beta\psi_1$ for a constant $\beta>0$ then $\tilde{P}_1\geq\beta\bar{P}_1$ if $\beta>1$, $\tilde{P}_1\leq\beta\bar{P}_1$ if $\beta<1$.
 \item{ii)} In the saturation and under saturation  cases, if  $\tilde{\psi}_1=\psi_1+\lambda$ for a constant $\lambda>0$ then $\tilde{P}_1=\bar{P}_1+\lambda M_1$.
 \end{description}
 \end{theorem}
 In Theorem \ref{new} we expand on case (i) of Theorem \ref{new0} and obtain the following, somewhat surprising result:

 \begin{theorem} \label{new}\
  Under Definition \ref{uni},
 \begin{description}
 \item{i)}
 Suppose  $\tilde{\psi}_1\geq \beta\psi_1$ where $\beta >1$ is a  constant.  Then
 \be\label{ini}\tilde{P}_1\geq (\beta-1)\bar{P}_1 \ . \ee
 \item{ii)} For any $\beta>2$, $s>\beta-1$  there exists such a system ($\vpsi,\vM$) and $\tilde{\psi}_1\geq \beta\psi_1$  verifying Assumption \ref{uni} such that
 $$\tilde{P}_1<s\bar{P}_1 \ . $$
 In particular, the inequality  (\ref{ini}) is sharp  in the case $\beta>2$.
 \end{description}
     \end{theorem}
     \begin{cor}
     The i.v of an agent cannot decrease if its  wisdom $\psi_i$ is {\em at least doubled} ($\tilde{\psi}_i\geq 2\psi_i$).
     \end{cor}
    In Theorem \ref{new1} we obtain sharp conditions for {\it decrease} of i.v, given an increase of the corresponding wisdom:
\begin{theorem}\label{new1}   Under Assumption \ref{uni}, in   the under saturation/saturation  cases:  \begin{description}
\item{i)}If $1<\beta <2$, $\lambda\geq 0$ and
 \be\label{c2}\beta\psi_1(x)\leq \tilde{\psi}_1(x)\leq \beta\psi_1(x)+\lambda \ ,\ee
 then
 \be\label{barP1} \tilde{P}_1 \geq \bar{P}_1 -\frac{M_1\lambda(2-\beta)}{\beta-1}\ . \ee
  \item{ii)} For any $1<\beta<2, \lambda>0, s<\lambda(2-\beta)/(\beta-1)$  there exists a system  $(\vpsi, \vM)$ and \\
  $\beta\psi_1\leq \tilde{\psi}_1\leq \beta\psi_1+\lambda$  verifying Assumption \ref{uni} such that
 $$\tilde{P}_1<\bar{P}_1-M_1 s \   \ . $$
  In particular, the inequality  (\ref{barP1}) is sharp.
 \end{description}
 \end{theorem}
 The proofs of Theorems \ref{new0}-\ref{new1} are given in Section \ref{duality}.
 \section{Coalition/Partition games}\label{ColPar}
    \subsection{Cartels and coalitions}
Suppose that
some  agents $J\subset I$ decide to establish  a cartel. In the free market interpretation it means that they coordinate the price of their services,  i.e.  $p_i:=p_J$  for some $p_J\in\R$ for any $i\in J$. In the big brother interpretation it means that $J$ is replaced by a single "agent" whose capacity is
\be\label{MJ}M_J:=\sum_{i\in J} M_i\ee
 and its wisdom is the maximal wisdom  of its agents, namely
\be\label{psiJ}\psi_J(x):= \max_{i\in J}\psi_i(x) \ \ \ \forall x\in X \  \ . \  \ee
More generally, let ${\cal J}:=\{J_1, \ldots J_k\}$ where $J_i\in 2^I$ a disjoint covering of $I$, that is
$$I=\cup_{i=1}^k J_i \ , \ \ \ \ J_i\cap J_l=\emptyset \ \ \ \forall i\not=l\in \{1, \ldots k\} \  $$
Let $M_{\cal J}:=\{ M_{J_1}, \ldots M_{J_k}\}$, $\psi_{\cal{J}}:=\{ \psi_{J_1}, \ldots \psi_{J_k}\}$. Let $\vec{1}^{(i)}\in\R^N$ given by $1^{(i)}_j=1$ if $j\in J_i$, $1^{(i)}_j=0$ if $j\not\in J_i$. Let $\R_{\cal{J}}:= \text{Span}\left\{ \vec{1}^{(1)}, \ \ldots  \vec{1}^{(k)}\right\}\subset \R^N$.
 In analogy to (\ref{defSigma}) we set
\be\label{defSigmaJ}\Sigma_{\psi_{\cal J}}(\vM_{\cal J}):=
 \max_{\vA\in {\cal P}_{\vM_{\cal J}}} \sum_{J\in {\cal J}}P_{\psi_J}(A_J) \ .  \  \ee
\be\label{XiJ}\Xi^0_{\vpsi_{\cal J}}(\vp):=\int_X\max_{J\in {\cal J}}[\psi_J(x)-p_J]_+d\mu(x):\R^k\rightarrow\R  \ . \ee
\begin{prop}\label{propcoal}
Assume the conditions of Theorem \ref{old}. Given a partition $\cal{J}$ of $I$ as above, we get  in the US,S cases
$$ \Sigma_{\vpsi_{\cal J}}(\vM_{\cal J}) = \min_{\vp\in \R_{\cal{J}}}\Xi^0_{\vpsi_{\cal J}}(\vp)+\vp\cdot\vM_{\cal J} \ . $$
where $\Xi^0_{\vpsi_J}$ as defined in (\ref{XiJ}). In the OS case,
$$ \Sigma_{\vpsi_{\cal J}}(\vM_{\cal J}) = \max_{\vm\leq \vM}\min_{\vp\in \R_{\cal{J}}}\Xi^0_{\vpsi_{\cal J}}(\vp)+\vp\cdot \vM_{\cal J} \ . $$
 In all cases there exists a unique equilibrium price vector $\vp_0\in\R_{\cal J}$, and the corresponding partition is the unique partition verifying
 (\ref{defSigmaJ}).
\end{prop}
\begin{proof}
First note that
$$ \{ x\in X; \psi_{J}(x)-\psi_{J^{'}}(x)=r \}\subset \cup_{j\in J, i\in J^{'}} \{ x\in X; \psi_j(x)-\psi_i(x)=r\}$$
for any $J, J^{'}\in \vec{J}$ so
 (\ref{post}, \ref{post0}) imply, for any $r\in\R$,
 $$\mu(x; \psi_J(x)-\psi_{J^{'}}(x)=r)=0 \ \ ; \ \ \ \mu(x; \psi_J(x)=r)=0  \ \ \forall  \ J,  \ J^{'}\in {\cal J} \ . $$
Hence, the conditions of Theorem \ref{old} hold for this modified setting. The only thing left to show is that for each $\vp\in \R_{\cal{J}}$, the associated partition $\vA^+(\vp)$  (c.f. \ref{A+0}, \ref{A+i}) is in ${\cal P}_{M_{\cal J}}$. This follows immediately from the following trivial  identity: if $\vp:= p_{J_1}\vec{1}^{(1)}+ \ldots +p_{J_k}\vec{1}^{(k)}\in \R_{\cal{J}}$ then
$$ \max_{j\in I}[\psi_j(x)-p_j]_+\equiv \max_{J\in {\cal J}}[\psi_J(x)-p_J]_+ \ . $$
Thus, the Proposition follows from Theorem \ref{old} applied to this modified system.
\end{proof}
Note that Proposition \ref{propcoal} does not say anything about the partition {\em inside} a coalition $J$. In fact, it determines only the set of consumers
$\bar{A}^+_J$ of this coalition. In particular, the  joint value of the coalition $J$ is
$$ \bar{P}_J(\vM, \vpsi):= \int_{\bar{A}^+_J}\psi_Jd\mu \ . $$
Do the agents of this coalition benefit from their collaboration? A necessary condition for this is that the combined  value of the coalition  $J$  is, at least, not smaller  than the sum of the values of its members standing alone, namely
\be\label{motive} \sum_{i\in J} \bar{P}_{\{i\}}(\vM, \vpsi)\leq \bar{P}_J(\vM, \vpsi) \  \ \ ? \ee
where $\bar{P}_{\{i\}}(\vM, \vpsi)$ is the i.v of $i\in J$ determined by the assumption that the coalition $J$ is disintegrated and its members are competing as individuals.  In order to test the inequality, suppose $N=3$ and $J_1=\{1,2\}$, $J_2=\{3\}$. Suppose further that $M_2$ is very small, so $\bar{P}_2(\vM, \vpsi)\approx 0$ while the i.v of agent 1, standing alone, is the same as her i.v if she competes with agent 3 only, and her wisdom is $\psi_1$.\par
On the other hand, the i.v. of the coalition $J$ is approximated by the i.v of agent 1 playing against agent 3, where now her wisdom is, by Proposition \ref{propcoal},  $\psi_{1,2}:= \max(\psi_1, \psi_2)$. Hence (\ref{motive}) is reduced into
\be\label{contra0}\bar{P}_{\{1\}}(M_1,M_3,\psi_1,\psi_3) \leq \bar{P}_{\{1\}}(M_1,M_3, \psi_{1,2}, \psi_3) \ \ ? \ee
Since $\psi_{1,2}\geq \psi_1$ it raises the question of {\em monotonicity of i.v} with respect to their wisdom. We may question this intuition, based on  Theorem \ref{new1} in Section \ref{indipro}.

  \subsection{Review on cooperative games}\label{gamerev}
  A cooperative game is a game where groups of players ("coalitions") may enforce cooperative behavior,  hence the game is a competition between coalitions of players, rather than between individual players.
  \par
  This section is based on the monograph \cite{Gi}.
  \begin{defi}\label{deficoal0}
  A {\it cooperative  game} (CG game) in $I:=\{1, \ldots N\}$ is given  by a {\em reward function} $\nu$  on the subsets of $I$:
$$ \nu: 2^I\rightarrow \R \ \ , \ \ \nu(\emptyset)=0 \ . $$
\end{defi}
  \begin{defi}
The {\it core} of a game $\nu:2^I\rightarrow\R_+$,
 $Core(\nu)$,  is a set of vectors $\vx:=(x_1, \ldots x_N)\in \R^N$ which satisfy the following conditions
\be\label{fs1} \sum_{i\in I} x_i\leq \nu(I)  \ee
and
\be\label{fs2} \forall J\subseteq I, \ \ \sum_{i\in J} x_j\geq \nu(J) \ . \ee
The {\it grand coalition} $J=I$ is \underline{\em stable} if the core is not empty.
\end{defi}
 Non-emptiness of the
core guarantees that no sub-coalition $J$ of the grand coalition $I$ will be formed. Indeed, if such a sub-coalition $J$ is formed, its reward $\nu(J)$ is not larger than the sum of the rewards of its members, guaranteed by the grand coalition.

In many cases, however, the core is empty.


We can easily find a necessary condition for the core to be non-empty. Suppose we divide $I$ into a  set of coalitions $J_k\subset I$ , $k=1, \ldots m$  such that $J_k\cap J_{k^{'}}=\emptyset$ for $k\not= k^{'}$ and $\cup_{k=1}^m J_j=I$.
\begin{prop}\label{pncI}
 For any such division, the condition
\be\label{ncI} \sum_{k=1}^m \nu(J_k)\leq \nu(I)\ee
is  necessary  for the grand coalition to be stable.
\end{prop}
\begin{proof}
Suppose $\vx\in Core(\nu)$.  Let $\tilde{\nu}(J):= \sum_{i\in J} x_i$. Then $\tilde{\nu}(J)\geq \nu(J)$ for any $J\subseteq I$. If (\ref{ncI}) is violated for some division $\{J_1, \ldots J_m\}$, then $\sum_{k=1}^m\tilde{\nu}(J_k)\geq \sum_{k=1}^m\nu(J_k) >\nu(I)$. On the other hand, $\sum_{k=1}^m\tilde{\nu}(J_k)=\sum_{i\in I}x_i\leq \nu(I)$, so we get a contradiction.
\end{proof}
Note that {\em super-additivity}
\be\label{sad} \nu(J_1)+\nu(J_2)\leq \nu(J_1\cup J_2) \ \ \ \forall \ J_1\cap J_2=\emptyset \ee
is a sufficient condition for (\ref{ncI}). However, (\ref{sad}) by itself is {\em not} a sufficient condition for   the stability of the grand coalition.
\begin{example}
In case $N=3$ the  game $\nu(1)=\nu(2)=\nu(3)=0$, $\nu(12)=\nu(23)=\nu(13)=3/4$, $\nu(123)=1$ is super-additive but its core is empty.
\end{example}
We may extend condition (\ref{ncI}) as follows: A {\it weak division} is a function $\lambda: 2^I\rightarrow \R$ which satisfies the following:
\begin{description}
\item {i)} \ For any $J\subseteq I$, $\lambda(J)\geq 0$.
\item{ii)} \ For any $i\in I$, $\sum_{J\subseteq I; i\in J} \lambda(J)=1$.
\end{description}
A collection of such sets $\{J\subset I; \lambda(J)>0\}$ verifying (i,ii) is called  a \underline{\em balanced collection} \cite{Gi}.
\par
We can think about $\lambda(J)$ as the probability of the coalition $J$. Note that any division $\{J_1, \ldots J_m\}$ is, in particular, a weak division where $\lambda(J)=1$ if $J\in \{J_1, \ldots  J_m\}$, and $\lambda(J)=0$ otherwise.
\par
It is not difficult to extend the necessary condition (\ref{ncI}) to weak subdivisions as follows:
\begin{prop}\label{pnscI}
For any weak subdivision $\lambda$, the condition
\be\label{nscI} \sum_{J\in 2^I} \lambda(J)\nu(J)\leq \nu(I) \  \ee
is  necessary  for the grand coalition to be stable.
\end{prop}
The proof of Proposition \ref{pnscI} is a slight modification of the proof of Proposition \ref{pncI}.

However, (\ref{nscI}) is also a sufficient condition for the stability of the grand coalition $I$. This is the content of Bondareva-Shapley Theorem \cite{bon}, \cite{shap1}, \cite{shap2}:
 \begin{theorem}\label{shaply}
 The  grand coalition  is stable if and only if it satisfies (\ref{nscI}) for any weak division $\lambda$.
 \end{theorem}
The condition of Theorem \ref{shaply} is easily verified for super-additive game in case $N=3$.
\begin{cor}\label{corrrrr}A super additive cooperative game of 3 agents ($N=3$) admits a non-empty core iff
\be\label{onlytp} \nu(12)+\nu(13)+\nu(23)<2\nu(123) \ . \ee
\end{cor}
Indeed, it can easily be shown that all weak subdivision for  $N=3$ are spanned by
$$\lambda(J)=1/2 \ \ \text{if} \ \ J=(12), (13), (23) \ \ \ ; \ \ \ \lambda(J)=0 \ \ \text{otherwise}  \ , $$
 and the trivial ones.

\subsection{Back to  coalition/partition game}
We define a cooperative game for the $N$ agents which is based on the following:
\begin{itemize}
\item Let $\vpsi$  be a wisdom vector  verifying Assumption \ref{funda}, and $\vM$ a capacity vector in saturation (i.e $\sum_{i\in I}M_i=\mu(X)$).
\item
Let $J\subseteq I$ be a coalition, and $J^{-}:=I-J$ is the complementary coalition.
\item
Consider the two superagents $J, J^{-}$ subjected to the capacities $M_J, M_{J^-}$ and wisdoms $\psi_J, \psi_{J^{-}}$ as defined in  (\ref{MJ}, \ref{psiJ}).
\item
By Theorem \ref{old} and Proposition \ref{propcoal} there exists a unique maximizer $\bA_J, \bA_{J^{-}}$ of
$$(A_J, A_{J^{-}})\in {\cal P}_{M_J, M_{\Jm}}\mapsto \int_{A_J}\psi_Jd\mu + \int_{A_{J^{-}}}\psi_{\Jm}  d\mu \ .  $$
 \item Let $\bar{P}_J:=\int_{\bar{A}_J}\psi_Jd\mu$ be the i.v of the coalition $J$ under optimal partition.
  \begin{defi}\label{deficoal}
  A {\it coalition/partition  (CP) game} is given  by a function on the subsets of $I$:
$$ \nu: 2^I\rightarrow \R_+ \ \ ; \ \ \   \nu(J):= \bar{P}_J \ . $$
In particular, $\nu(I)=\int_X\psi_Id\mu$, and $\nu(\emptyset)=0$.
\end{defi}
\end{itemize}

  Note that the CP game satisfies the following condition: For each $J\subset I$,
\be\label{nc1}\nu(J) + \nu(J^-)\leq \nu(I) \ \ \forall \ \ J\subset I \ , \ee
which is  a necessary condition for super-additivity (\ref{sad}).

However, a  CP game is not super-additive in general. In fact, the inequality (\ref{contra0})  may be violated, as  we know from Theorem \ref{new1}.

There is a special case, introduced in Example \ref{exlp}, Section \ref{firstsec},  for which we can guarantee super-additivity for PC game:
\begin{assumption}\label{asspsi}
There exists non-negative $\psi:X\in C(X)$ satisfying $\mu(x; \psi(x)=r)=0$ for any $r\in \R$. The wisdoms $\psi_i$ are given by $\psi_i=\lambda_i\psi$ where  $\vec{\lambda}:=(\lambda_1, \ldots \lambda_N)\in \R^N_+$ such that $0<\lambda_1<\ldots <\lambda_N$. We further assume the saturation case $\sum_{i=1}^N M_i=\mu(X)$.
\end{assumption}
\begin{prop}\label{propmulti}
 Under assumption \ref{asspsi}, for any $\vM\in\R^N_+$, the corresponding CP game is super-additive.
\end{prop}
\begin{proof}
From Example \ref{exlp}
 (in particular from (\ref{unidi})) we obtain that the i.v. of agent $i$ under optimal partition is

 \be\label{1Div}\bar{P}_i\equiv \lambda_i\left(F^*({\cal M}_{i+1})-F^*({\cal M}_{i})\right) \ . \ee

Note that $F^*$ is convex on $[0,1]$ by Example \ref{exlp}. In addition, $F^*(0)=0$.

It follows that for any $\alpha, \beta>0$ such that $\alpha+\beta\leq 1$:
\be\label{FF*} F^*(\alpha+\beta)\geq F^*(\alpha) + F^*(\beta) \ . \ee
Indeed, since $F^*$ is convex
\be\label{F*}F^*(\alpha+\beta)\geq F^*(\alpha) + \beta (F^*)^{'}(\alpha)\  . \ee By the same argument
\be\label{F**}0=F^*(0)\geq F^*(\beta) - \beta(F^*)^{'}(\beta) \ . \ee
Assuming, with no  loss of generality, that $\alpha\geq \beta$,  we get (\ref{FF*}) from the convexity of $F^*$ which implies  $(F^*)^{'}(\alpha)\geq  (F^*)^{'}(\beta)$, together with (\ref{F*}, \ref{F**}).
\par
By the definition of the CP game and (\ref{1Div}) we obtain that
\begin{multline}\label{nunu} \nu(J)= \lambda_N(F^*(\mu(X))-F^*(\mu(X)-M_J)) \ \ \text{if} \ \ \{N\}\in J \ \ \ ; \ \ \\ \nu(J)=\lambda_J(F^*(M_J)-F^*(0))\equiv \lambda_JF^*(M_J) \ \ \text{if} \ \{N\}\not\in J \ . \end{multline}
where $\lambda_J:=\max_{i\in J}\lambda_i$.

Let now  $J_1, J_2\subset I$ such that    $J_1\cap J_2=\emptyset$ (in particular, $M_{J_1}+M_{J_2}\leq \mu(X)$ since $\vM$ is a saturated vector).

\par
Assume first $\{N\}\not\in J_1\cup J_2$. Then from (\ref{nunu}, \ref{FF*})
\begin{multline} \nu(J_1\cup J_2)=\max(\lambda_{J_1}, \lambda_{J_2})F^*(M_{J_1\cup J_2})=  \max(\lambda_{J_1}, \lambda_{J_2})
F^*(M_{J_1}+M_{J_2}) \\ \geq \lambda_{J_1}F^*(M_{J_1})+\lambda_{J_2}F^*(M_{J_2}) = \nu(J_1)+\nu(J_2) \ . \end{multline}
Next, if, say,  $\{N\}\in J_1$ then
$$ \nu(J_1\cup J_2)= \lambda_N(F^*(\mu(X))-F^{*}(\mu(X)-M_{J_1}+M_{J_2})) \ \ , \ \ \nu(J_1)=\lambda_N(F^*(\mu(X))-F^{*}(\mu(X)-M_{J_1}))\ \ ,$$
$$ \nu(J_2)=\lambda_{J_2}F^{*}(M_{J_2})\ , $$
so
$$ \nu(J_1\cup J_2)-\nu(J_1)-\nu(J_2)\geq \lambda_N\left[F^*(1-M_{J_1})-F^*(M_{J_2}) - F^*(1-(M_{J_1}+M_{J_2})\right]\geq 0  \ , $$
again, by (\ref{FF*}).
\end{proof}

Under the  assumption of Proposition \ref{propmulti} we may guess, intuitively,  that the grand coalition is stable if  the  gap between the wisdoms of the agents is   sufficiently large (so the other agents are motivated to join the smartest one), and the capacity of the wisest agent ($N$) is sufficiently small (so (s)he is motivated to join the others as well).  Below we prove this intuition in the case  $N=3$:
\begin{prop}\label{example1}
Under the assumption of Proposition \ref{propmulti} and $N=3$,
$$  \frac{\lambda_3}{\lambda_2}> \frac{F^{*}(M_1+M_2)}{F^*(M_2)+F^*(M_1)}$$
is a necessary and sufficient for the stability of the grand coalition (i.e the non-emptiness  of a core). Here $F^*$ is as defined in Example \ref{exlp} (Section \ref{firstsec}).
\end{prop}

\begin{proof}
From  Corollary \ref{corrrrr} and Proposition \ref{propmulti} we have only to prove  (\ref{onlytp}). Now,
$$ \nu(123)=\lambda_3F^*(\mu(X)) \ \ , \nu(13)=\lambda_3(F^*(\mu(X))-F^*(M_2)), \ \  \nu(23)=\lambda_3(F^*(\mu(X))-F^*(M_1))$$ and  $$\nu(12)=\lambda_2F^*(M_1+M_2)  \ . $$
The result follows from substituting the above in (\ref{onlytp}).

\end{proof}

\section{Free Market vs.  Big Brother: Duality}\label{duality}
Here we verify the dual nature of the free-market/big brother formulation and prove Theorem \ref{old}. This duality is also the main tool for proving Theorems \ref{new}, \ref{new1}. The key Lemmas for all these results are the following:
\begin{lemma}\label{lemc1}
Let
 \be\label{Xino0}\Xi_{\vpsi}(\vp):=\int_X\max_{i\in I}[\psi_i(x)-p_i]d\mu(x) \ . \ee
 where $\vpsi:=\{\psi_1, \ldots \psi_N\}$ satisfies  (\ref{post}).      Then $\Xi_{\vpsi}$  is convex  on $\R^N$.
Moreover, it is differentiable at any point and
\be\label{condiff}  \frac{\partial \Xi_{\vpsi}}{\partial p_i}(\vp) = -\mu(A_i(\vp)) \  . \ \ \ee
  Here
\be\label{Aip}A_i(\vp):= \{x\in X; \ \psi_i(x)-p_i<\psi_j(x)-p_j \ \ \forall j\not= i\} \ . \ee
\end{lemma}
\begin{proof}
Note that $A_i(\vp)$ are mutually disjoint and, by (\ref{post}),  $\mu(\cup_{i\in I} A_i(\vp))=\mu(X)$. \\
Consider $\xi:X\times\R^N\rightarrow \R$ give by $\xi(x,\vp)=\max_{i\in I}[\psi_i(x)-p_i]$. Then $\xi$ is convex on $\R^N$ for any $x\in X$. Moreover,
$$ \frac{\partial \xi}{\partial p_i}=\left\{ \begin{array}{cc}
                                               -1  \  if & x\in A_i(\vp) \\
                                               0  \ if & \exists j\not= i, \ x\in A_j(\vp)
                                             \end{array}\right. \  $$
                                             In particular, the $\vp$ derivatives of $\xi$ exists $\mu$ a.e in $X$, $\nabla_{\vp}\xi\in \mathbb{L}_1(X)$ for any $\vp\in\R^N$ and are uniformly integrable.
Since $\Xi_{\vpsi}:=\int_X\xi d\mu$ by definition, it is still convex in $\R^N$, its derivatives exists  everywhere
and $\nabla_{\vp}\Xi_{\vpsi}=\int\nabla_{\vp}\xi(x,\vp) d\mu(x)$.  \\

\end{proof}

\begin{lemma}\label{lemc2}
Let $a>0$ and $\vpsi:=\vpsi(x,t):X\times [0,a]\rightarrow \R^N_+$  for any $t\in [0,a]$.  Assume further that each component $t \mapsto\psi_i(x,t)$ is convex and differentiable for any $x\in\R^N$ and $$\partial_t\psi_i := \dot{\psi_i}\in \mathbb{L}^\infty(X\times [0,a])$$ for any $i\in I$. Then the functions
$$(\vp,t) \mapsto \Xi(\vp, t):= \int_X\max_{i\in I}[\psi_i(x,t)-p_i]d\mu(x)$$ is convex  on $\R^N\times [0,a]$, and, {\em if} its $t$ derivative  $\dot{\Xi}:= \partial_t\Xi$ exists at $( \vp,t)$ then
 \be\label{17}
 \dot{\Xi}(p,t) = \sum_{i\in I}\int_{A_i(\vp,t)}\dot{\psi}_i(x,t) d\mu\ee
  Here
 \be\label{Aipt}A_i(\vp,t):= \{x\in X; \ \psi_i(x,t)-p_i>\psi_j(x,t)-p_j \ \ \forall j\not= i\} \ . \ee
\end{lemma}
\begin{proof} The proof  is  basically the same as   the proof of Lemma \ref{lemc1}. Here we define
$\xi:X\times\R^N\times[0,a]\rightarrow \R$ give by $\xi(x,\vp,t)=\max_{i\in I}[\psi_i(x,t)-p_i]$, and $\Xi(\vp,t):=\int_X\xi(x,\vp,t)d\mu(x)$. Again, $\xi$ is convex on $\R^N\times[0,a]$ for any $x\in X$ so $\Xi$ is convex on $\R^N\times[0,a]$ as well, while
$$ \dot{\xi}=\left\{ \begin{array}{cc}
                                               \dot{\psi}_i(x, t) \  \  \ if  & x\in A_i(\vp,t) \\
                                               0  \ \  \ if& \exists j\not= i, \ x\in A_j(\vp,t)
                                             \end{array}\right.  \  \ \  $$
                                             implies (\ref{17}).
\end{proof}
\subsection{Proof of Theorem \ref{old}}
Note  the difference between $\Xi_{\vpsi}$ (\ref{Xino0}) and $\Xi_{\vpsi}^0$ (\ref{Xi0}). Observe that \be\label{additive}\Xi_{\vpsi}(\vp +\alpha\vec{1})=\Xi_{\vpsi}(\vp)-\alpha\mu(X) \ \ \ ; \ \ \ \vec{1}:=(1, \ldots 1)\in\R^N\ . \ee
In particular $\nabla\Xi_{\vpsi}(\vp)=\nabla\Xi_{\vpsi}(\vp+\gamma\vec{1})$  and,  in the saturated case $\vM\cdot\vec{1}=\mu(X)$:
$$ \Xi_{\vpsi}(\vp)+\vM\cdot\vp=  \Xi_{\vpsi}(\vp+\gamma\vec{1})+\vM\cdot(\vp+\gamma\vec{1})$$
for any $\gamma \in\R$.
\begin{remark}\label{rem51}
In the saturation case it follows that $\vp\mapsto \Xi_{\vpsi}(\vp)+\vp\cdot\vM$ is invariant under the shift. In addition we observe that, for $-\gamma$ large enough, such that $A_0^+(\vp+\gamma\vec{1})=\emptyset$,
$$ \Xi^0_{\vpsi}(\vp+\gamma \vec{1})=\Xi_{\vpsi}(\vp+\gamma \vec{1})$$
 (recall Remark \ref{remhadash}).  So, in the saturation case,  we may replace $\Xi^0_{\vpsi}$ by $\Xi_{\vpsi}$ and restrict the domain of $\Xi_{\vpsi}$ to
\be\label{normal} \vp\in \R^N \ \ , \ \ \vp\cdot\vec{1}=0 \ .  \ee
\end{remark}
\begin{remark}\label{rem52}
We may actually  unify the under-saturated and saturated cases (US+S). Indeed, in the under-saturated case we add the "null agent" whose wisdom $\psi_0=0$ and whose capacity is $M_0:=\mu(X)-\sum_{i\in I}M_i$. In that case we restrict our domain $(p_0, p_1, \ldots p_N)\in \R^{N+1}$ to  $p_0=0$, and we get by definition (\ref{Xi0}, \ref{Xino0})
 $$ \Xi_{\vec{0,\psi}}(0, p_1, \ldots p_N)=\Xi_{\vpsi}^0(p_1, \ldots p_N) \ $$
  Hence the proof of  the first part (saturated case), given below, includes also the proof of the under-saturated case.  The over-saturated case will be treated later  on.
\end{remark}
Let $\vA$ be any partition in
$$ {\cal P}_{\vM}:= \{ \vA; \ \mu(A_i)= M_i  \ \ \text{and} \ \  \mu(A_i\cap A_j)=0 \ \ \text{if} \ \ i\not= j\in I \ . \}$$
 Then
 \begin{multline}\label{multdual}\Sigma_{\vpsi}(\vA):= \sum_{i\in I}\int_{A_i}\psi_id\mu= \sum_{i\in I}\int_{A_i}(\psi_i-p_i)d\mu + \vM\cdot\vp  \\ \leq
 \sum_{i\in I}\int_{A_i}\max_{i\in i}(\psi_i-p_i)d\mu + \vM\cdot\vp= \int_X\max_{i\in i}(\psi_i-p_i)d\mu + \vM\cdot\vp := \Xi_{\vpsi}(\vp) +\vM\cdot\vp\ . \end{multline}
 In particular
 \be\label{dual} \bar{\Sigma}_{\vpsi}(\vM):=\sup_{\vA\in{\cal P}_{\vM}} \Sigma_{\vpsi}(\vA) \leq \inf_{\vp\in\R^N} \Xi_{\vpsi}(\vp)+\vM\cdot\vp \ . \ee
 Suppose we prove the existence of a minimizer $\vp_0$ to the right side of (\ref{dual}). Then, by Lemma \ref{lemc1} (\ref{condiff}) we get
 $$\frac{\partial \Xi_{\vpsi}}{\partial p_i}(\vp_0)=-M_i \ \  \  i\in I  .$$
 Moreover,  $\vA(\vp_0)\in {\cal P}_{\vM}$ is an optimal partition so there is an equality in (\ref{dual}). Conversely, if $\vA\in {\cal P}_{\vM}$ is an optimal partition then there is an equality in (\ref{multdual}) so
 $$ \psi_i(x)-p_{0,i} = \max_{j\in I}(\psi_j(x)-p_{0,j})  $$
for $\mu$ a.e. $x\in A_i$. Hence $A_i\subseteq A_i(\vp_0)$ (up to a negligible $\mu$ set).  By (\ref{nl}) we know that $\mu(\cup_{i\in I} A_i)=\mu(X)$ so  $A_i= A_i(\vp_0)$, again up to a negligible $\mu$ set.
 In  particular it follows that an optimal partition $\vec{\bar{A}}\in {\cal P}_{\vM}$ for $\Sigma_{\vpsi}$ is unique
 \par
 We now prove the existence of such a minimizer $\vp_0$.  Let $\vp_n$ be a minimizing  sequence of  $\vp \mapsto\Xi_{\vpsi}(\vp)-\vp\cdot\vM$, that is
$$\lim_{n\rightarrow\infty} \Xi_{\vpsi}(\vp_n)+\vp_n\cdot\vM=\inf_{\vp\in\R^N}\Xi_{\vpsi}(\vp_n)+\vp_n\cdot\vM \ . $$

 Let $\|\vp\|_2:= (\sum_{i\in I}p^2_i)^{1/2}$ be the Euclidian norm of $\vp$. If we prove that for any minimizing sequence $\vp_n$ the norms $\|\vp_n\|_2$ are uniformly bounded, then there exists  a converging subsequence whose limit is the minimizer  $\vp_0$, and we are done. This follows, in particular, since $\Xi_{\vpsi}$ is a closed (lower-semi-continuous) function.
 \par
 Assume there exists a subsequence along which $\|\vp_n\|_2\rightarrow\infty$. Let $\hat{\vp}_n:= \vp_n/\|\vp_n\|_2$.
Then
 \begin{multline}\label{try1}\Xi_{\vpsi}(\vp_n)+\vp_n\cdot\vM:= \left[\Xi_{\vpsi}(\vp_n)-\vp_n\cdot\nabla_{\vp}\Xi_{\vpsi}(\vp_n)\right] + \vp_n\cdot\left( \nabla_{\vp}\Xi_{\vpsi}(\vp_n)+\vM\right)\\
 =\left[\Xi_{\vpsi}(\vp_n)-\vp_n\cdot\nabla_{\vp}\Xi_{\vpsi}(\vp_n)\right] + \|\vp_n\|_2\hat{\vp}_n\cdot\left( \nabla_{\vp}\Xi_{\vpsi}(\vp_n)+\vM\right) \ . \end{multline}
Note  that
\be\label{Ximu0} \Xi_{\vpsi}(\vp)-\vp\cdot\nabla\Xi_{\vpsi}(\vp)= \sum_{i\in I}\int_{A_i(\vp)} \psi_id\mu\ , \ee
so, in particular
 \be\label{try2}0 \leq \int_X \min_{i\in I} \psi_i d\mu \leq \left[\Xi_{\vpsi}(\vp)-\vp\cdot\nabla_{\vp}\Xi_{\vpsi}(\vp)\right]=\sum_{i\in I} \int_{A_i(\vp)}\psi_i(x)d\mu \\ \leq \int_X \max_{i\in I} \psi_i d\mu<\infty \ . \ee
 By (\ref{try1}- \ref{try2}) we obtain, for $\|\vp_n\|_2\rightarrow\infty$,
 \be\label{lim0phat}\lim_{n\rightarrow\infty} \hat{\vp}_n\cdot\left( \nabla_{\vp}\Xi_{\vpsi}(\vp_n)+\vM\right)=0 \ . \ee
   Since $\hat{\vp}_n$ lives in the unit sphere $S^{N-1}$  (which is a compact set), there exists a subsequence for which $\hat{\vp}_n\rightarrow \hat{\vp}_0:= (\hat{p}_{0,1}, \ldots \hat{p}_{0,N})\in S^{N-1}$. Let $P_-:= \min_{i\in I} \hat{p}_{i,0}$ and $J_-:= \{ i\in I \ ; \hat{p}_{0,i}=P_-\}$.

 Note that for $n\rightarrow\infty$ along such a subsequence, $p_{n,i}-p_{n, i\prime}\rightarrow\infty$ for $i\not\in J_-, i\prime\in J_-$. It follows that $A_{i}(\vp_n)=\emptyset$ if $i\not\in J_-$ for $n$ large enough, hence $\mu(\cup_{i\in J_-}A_i(\vp_n)) =\mu(X)=\mu(X)$ for $n$ large enough. Let
 $\mu_i^n$ be the restriction of $\mu$ to $A_i(\vp_n)$. Then the limit $\mu^n_i\rightharpoonup \mu_i$ exists (along a subsequence) where $n\rightarrow\infty$.  In particular, by Lemma \ref{lemc1}
 $$\lim_{n\rightarrow\infty} \frac{\partial\Xi_{\vpsi}}{\partial p_{n,i}}(\vp_n)=-\lim_{n\rightarrow\infty} \int_X d\mu_i^n= -\int_X d\mu_i$$
 while  $\mu_i\not=0$  only if $i\in J_-$, and $\sum_{i\in J_-}\mu_i=\mu$. Since $\hat{\vp}_{0,i}=P_-$ for $i\in J_-$ is the minimal value of the coordinates of $\hat{\vp}_0$, it follows that
 $$\lim_{n\rightarrow\infty} \hat{\vp}_n\cdot\left( \nabla_{\vp}\Xi_{\vpsi}(\vp_n)+\vM\right)=-\vM\cdot\hat{\vp}_0 +P_-\sum_{i\in J_-}\int_Xd\mu_i=-\vM\cdot\hat{\vp}_0 +P_- \ . $$
Now, by definition, $\vM\cdot \hat{\vp}_0>P_-$ unless $J_-=I$. In the last case we obtain a contradiction of (\ref{normal}) since it implies $\hat{\vp}_0=0$ which contradicts $\hat{\vp}_0\in S^{N-1}$. So, if $J_-$ is a proper subset of $I$ we obtain a contradiction to (\ref{lim0phat}).
Hence $\|\vp_n\|_2$ is bounded, and a minimizer $\vp_0$ exists. This
 concludes the proof of the Theorem for the saturation  and under saturation case (c.f. Remark \ref{rem52}).
\par
To complete the proof for the over-saturation case we will show that the minimizer of (\ref{maximum}) is unique. This follows from the {\it strict concavity} of the function $\vm\mapsto \Sigma_{\vpsi}(\vm)$ on the simplex
 \be\label{simplex}S_N:=\{\vm\in\R^N_+; \vm\cdot\vec{1}\leq\mu(X)\} \ . \ee
 The proof of this strict concavity is given at the Appendix (Section \ref{appendix}).
$\Box$
\subsection{Individual values: Proof of Theorems \ref{new0}, \ref{new}, \ref{new1}}
For these we will use Lemma \ref{lemc2} in addition to Lemma \ref{lemc1}.\par
The proof of Theorem \ref{new0} is the easiest:
\begin{proof} of Theorem \ref{new0}
\begin{description}
\item{i)} Let $\vt:=(t_1, \ldots t_N)\in \R^N$. Let
$\vt\otimes\vpsi(x):=(t_1\psi_1(x), \ldots t_N\psi_N(x))$.
Consider
\be\label{Xitm}\Xi(\vt, \vp):=\Xi_{\vt\otimes \vpsi}(\vp) \ . \ee
\par
By Lemma \ref{lemc2}, $(\vt, \vp) \mapsto\Xi$ is  convex on $\R^N\times \R^N$ (in particular, convex in $\vp$ for fixed $\vt$ and convex in $\vt$ for fixed $\vp$). In addition
\be\label{pvt} \partial_{t_i}{\Xi}(\vt,\vp)=\int_{A_i(\vp, \vt)}\psi_i d\mu\equiv t_i^{-1}P_i(\vt)\ee
where
\be\label{pvt1}A_i(\vp,\vt):= \{x\in X; \ t_i\psi_i(x)-p_i<t_j\psi_j(x)-p_j \ \ \forall j\not= 1\} \ ,\ee
whenever $ \partial_{t_i}{\Xi}$ exists. Let
\be\label{XiSigma}\Sigma(\vt, \vM):= \min_{\vp\in\R^N} \Xi(\vp, \vt)+\vM\cdot \vp \ . \ee
Recall that $\Sigma(\vt,\vM)=-\infty$ unless $\vM$ is saturated. Still, it is real valued, convex as a function of $\vt$ for any saturated $\vM$. If $\vM$ is over-saturated then
$$\Sigma(\vt, \vM):= \max_{\vm\leq \vM}\min_{\vp\in\R^N} \Xi(\vp, \vt)+\vm\cdot \vp$$
is convex in $\vt$ as well.

 Then
$$ \partial_{t_i}{\Sigma}=\int_{A_i(\vp_0, \vt)}\psi_i d\mu\equiv t_i^{-1}P(\vt)$$
holds, where $\vp_0:=\vp_0(\vM,\vt)$ is the unique equilibrium price vector (perhaps up to an additive constant) guaranteed by Theorem \ref{old} for the utility vector $\vt\otimes\vpsi$.
Since the derivative of a convex function on the line is monotone non-decreasing we get for  $\vt_{(\beta)}:=(\beta, 1;\ldots 1)$, $\beta>1$
$$  P_1(\vt_{(\beta)})/\beta\equiv \partial_{\beta}{\Sigma}(\vt_{(\beta)}, \vM)\geq \partial_{s}{\Sigma}(\vt_{(s=1)}, \vM)\equiv \bar{P}_1$$
where $\tilde{P}_1\equiv P_1(\vt_{(\beta)})$ by (\ref{pvt}, \ref{pvt1}). This completes the proof for saturated and over-saturated $\vM$. The case of under-saturated $\vM$ is included in the saturated case (c.f Remark \ref{rem52}).
\item{ii)}
If $\psi_1\rightarrow \psi_1+\lambda$ then the optimal partition in the saturated and under saturated cases is unchanged.  Then
$$\tilde{P}_1:=\int_{\bar{A}_1}(\psi_1+\lambda)d\mu = \int_{\bar{A}_1}\psi_1d\mu+ \lambda\int_{\bar{A}_1}d\mu=\bar{P}_1+\lambda M_1 \ . $$
\end{description}
\end{proof}

\begin{proof}  of Theorem \ref{new}
 \begin{description}
 \item{i)} Let $\sigma:=\tilde{\psi}_1-\beta\psi_1\geq 0$, $\alpha:=\beta-1\geq 0$ and  a function $\phi:[0,1] \mapsto\R$ satisfying
  \be\label{psi2}\phi(0)=\dot{\phi}(0)=0 \ \text{and}   \ \ddot{\phi}\geq 0 \ \ \text{for any} \  t\in [0,1] \ \ , \ \phi(1)=1  \ .  \ee
  Define
 \be\label{psi1}\psi(x,t):= (1+\alpha t)\psi_1(x) + \sigma(x)\phi(t) \ . \ee
 So
 \be\label{1eqtilde}\psi(x,1)=\tilde{\psi}_1(x) \  \ee
   and  $\psi$ is convex in $t\in[0,1]$ for any $x$. Also
 $\dot{\psi}(x,t)=\alpha\psi_1(x)+\sigma(x)\dot{\phi}(t)$. Let now $\delta>0$. Then
 \be\label{newpsi}\psi(x,1)\geq \dot{\psi}(x,1)-\delta\|\sigma\|_\infty\ee
  provided
  \be\label{z1} \sigma(x)\dot{\phi}(1)\leq \sigma(x)+\psi_1(x)+\delta\|\sigma\|_\infty \ . \ee
 Since $\psi_1$ and $\sigma$ are non-negative, the later is guaranteed if $\dot{\phi}(1)\leq 1+\delta$. So, we choose
 $\phi(t):= t^{1+\eps}$ for some $\eps\in(0, \delta]$. This meets (\ref{psi2},\ref{z1}).

 Let now $\Sigma(\vM, t):= \Sigma_{\psi(, t), \psi_2, \ldots \psi_N}(\vM)$.  By Theorem \ref{old},
 $$\Sigma(\vM,t):=\inf_{\vp\in\R^N}\Xi_{\vpsi(\cdot, t)}(\vp) +\vp\cdot\vM$$
 in the US, S cases
 and  by Lemma \ref{lemc2}, $(\vp, t)\mapsto \Xi_{\vpsi(\cdot, t)}(\vp)$ is convex. So  $\Sigma$ is convex in $t$.
 In the OS case
 $$\Sigma(\vM,t):=\sup_{\vm\leq \vM}\inf_{\vp\in\R^N}\Xi_{\vpsi(\cdot, t)}(\vp) +\vp\cdot\vm$$
 is convex (as maximum of convex functions) as well. By the same Lemma
 \be\label{verified} \dot{\Sigma}(\vM, 0)=\int_{\bar{A}_1(0)}\dot{\psi}(,0)d\mu=\alpha\int_{\bar{A}_1(0)}\psi_1d\mu\equiv \alpha\bar{P}_1\ee
 where $\bar{A}_1(0)$ is the first component in the optimal partition associated with $\vpsi$, while, at $t=1$ we obtain from convexity and (\ref{newpsi})
 \be\label{alpha} \dot{\Sigma}(\vM, 1)=\int_{\bar{A}_1(1)} \dot{\psi}(x,1)d\mu \leq \int_{\bar{A}_1(1)}( \psi(x,1)+\delta\|\sigma\|_\infty)d\mu \leq \int_{\bar{A}_1(1)} \psi(x,1)d\mu +\delta\|\sigma\|_\infty \ \ee
 where $\bar{A}_1(1)$ is the first component in the optimal partition associated with \\ $\vpsi(1):= (\psi(, 1), \psi_2, \ldots \psi_N)$.  Since $\tau \mapsto \phi(\tau)$ is convex, $\tau\mapsto \Sigma(\vM,\tau)$ is convex as well by Lemma \ref{lemc2} and we get
 \be\label{beta} \dot{\Sigma}(\vM,1)\geq  \dot{\Sigma}(\vM,0) \ . \ee
 From (\ref{alpha}, \ref{beta})
 $$ \int_{\bar{A}_1(1)} \psi(x,1)d\mu \geq \alpha\bar{P}_1-\delta \|\sigma\|_\infty \ . $$

 Now, recall  $\beta:=1+\alpha$ and  $\psi(x,1):=\tilde{\psi}_1$ by (\ref{1eqtilde}), so $\int_{\bar{A}_1(1)} \psi(x,1)d\mu\equiv \tilde{P}_1$.  Since $\delta>0$ is arbitrary small, we obtain the  result.
 \item{ii)}
  Assume $N=2$, $M_1+M_2=\mu(X)$. We show the existence of non-negative, continuous  $\psi_1, \psi_2$, $x_1, x_2\in X$ and $\lambda>0$    such that, for given $\delta>0$
    \begin{description}
    \item{a)} $\Delta(x):= \psi_1(x)-\psi_2(x)< \Delta(x_1)$ for any $x\in X-\{x_1\}$.
    \item{b)}  $\Delta_\beta(x):= \beta\psi_1(x)-\psi_2(x)< \Delta_\beta(x_1)$ for any $x\in X-\{x_1\}$.
    \item{c)}   $\Delta_\beta(x_2)+\lambda =\Delta_\beta(x_1)+\delta$.
    \end{description}
    We show that  (a-c) is consistent with
    \be\label{s<}s\psi_1(x_1)>\beta\psi_1(x_2)+\lambda\ee
    for given $s>\beta-1$. \par
    Suppose (\ref{s<}) is  verified.  Let
    \be\label{theta0}\theta_0:= \left\{\begin{array}{cc}
                   1-\frac{|x-x_2|}{\eps} & \text{if} \ |x-x_2|\leq \eps \\
                   0 & \text{if} \ |x-x_2|> \eps
                 \end{array}\right.\ee
                 (assuming, for simplicity, that $X$ is a real interval).
                 Set $\tilde{\psi}_1:=\beta\psi_1+\lambda\theta_0$.  If $\eps$ is small enough then
    $\tilde{\psi}_1-\psi_2$ is maximized at $x_2$ by (b,c), while $\psi_1-\psi_2$ is maximized at $x_1$ by (a).
    Letting  $M_1<<1$ we find, by Example \ref{exAlambda}, that
    $\bar{P}_1\approx M_1\psi_1(x_1)$ and $\tilde{P_1}\approx M_1(\beta\psi_1(x_2)+\lambda\theta_0(x_2))= M_1(\beta\psi_1(x_2)+\lambda)$.
     By (\ref{s<}) we obtain the result.
     \par
     So, we have only to prove that (\ref{s<}) is consistent with (a-c). We rewrite it as
     $$\frac{s}{\beta-1} \left[ \Delta_\beta(x_1)-\Delta(x_1)\right] > \frac{\beta}{\beta-1} \left[ \Delta_\beta(x_2)-\Delta(x_2)\right]+\lambda \ . $$
    From  (c) we obtain
    $$\frac{s}{\beta-1} \left[ \Delta_\beta(x_1)-\Delta(x_1)\right] > \frac{\beta}{\beta-1} \left[ \Delta_\beta(x_2)-\Delta(x_2)\right]+\Delta_\beta(x_1)-\Delta_\beta(x_2)+\delta \ , $$
    that is
    \be\label{cucu} (s-\beta+1)\Delta_\beta(x_1)-\Delta_\beta(x_2)> (s-\beta)\Delta(x_2)+s(\Delta(x_1)-\Delta(x_2))+\delta(\beta-1) \ . \ee
    We now set $\Delta_\beta(x_1)$ and $\lambda$ large enough, keeping $\delta, \Delta_\beta(x_2),\Delta(x_1), \Delta(x_2)$ fixed. Evidently, we can do it such that (c) is preserved.
    Since $s-\beta+1>0$ by assumption, we can get  (\ref{cucu}).

 \end{description}
 \end{proof}
 \begin{proof} of Theorem \ref{new1}. \\
 \begin{description}
 \item{i)} Let $\beta=1+t$ where $t\in(0,1)$.
 We change (\ref{psi1}) into   \be\label{psi11}\psi(x,t):= (1+ t)(\psi_1(x)+\gamma) + \sigma(x)\phi(t)\ee
 and
 \be\label{magilush} \tilde{\psi}_1:= (1+t)\psi_1+\sigma\phi(t) \  \ee
 where $\gamma>0$ is a constant and $\sigma\geq 0$  on $X$.   Then
 $\dot{\psi}(x,t)=\psi_1(x)+\gamma+\sigma(x)\dot{\phi}(t)$, and we obtain
 \be\label{yomtov}\psi(x,t)\geq \dot{\psi}(x,t), \ \ t>0; \ \ \dot{\psi}(x,0)=\psi_1(x)+\gamma  \ee
 provided
  \be\label{z11} \sigma(x)\dot{\phi}(t)\leq \sigma(x)\phi(t)+t(\psi_1(x)+\gamma) \ ; \ \dot{\phi}(0)=0 \ .  \ee
 Since $\psi_1,\sigma$ are non-negative, the later is guaranteed if
 \be\label{z2} \dot{\phi}(t)\leq \phi(t)+\frac{t\gamma}{\|\sigma\|_\infty} \ ; \ \dot{\phi}(0)=0 \ .  \ee
   Since $t<1$, the choice $\phi(\tau):=\tau^{1+\eps}$ for $0\leq \tau\leq t$ and $\eps>0$ small enough (depending on $t$) verifies (\ref{z2}) provided
 \be\label{uuu}\|\sigma\|_\infty < \gamma t/(1-t) \ . \ee
 \par

  Hence we can let $\sigma$ to be any function verifying   (\ref{uuu}). Then  (\ref{psi11}, \ref{magilush}) imply
  \be\label{zirgug}(1+ t)\psi_1(x) \leq \tilde{\psi}_1(x)\leq (1+ t)\psi_1(x)+\frac{\gamma t^{2+\eps}}{1-t} \ . \ee
  Now, we note from the second part of (\ref{yomtov}) that
  \be\label{verified1} \dot{\Sigma}(\vM, 0)=\int_{\bar{A}_1(0)}\dot{\psi}(,0)d\mu=\int_{\bar{A}_1(0)}(\psi_1+\gamma)d\mu\equiv \bar{P}_1+\gamma M_1 \  \ee
 since $A_1(0)$ is independent of $\gamma$ in the S, US cases. In addition,  (\ref{psi11}, \ref{magilush},\ref{z2}) imply
  $$ \dot{\Sigma}(\vM, t)=\int_{\bar{A}_1(t)}\dot{\psi}(,t)d\mu\leq \int_{\bar{A}_1(t)}\psi(\cdot,t)d\mu= \int_{\bar{A}_1(t)}(\tilde{\psi}_1+(1+t)\gamma)d\mu \equiv \tilde{P}_1+(1+t)\gamma M_1 \ , $$
   where $\bar{A}_1(t)$ is the first component in the optimal partition associated with \\ $\vpsi(t):= (\psi(, t), \psi_2, \ldots \psi_N)$.  Since $\tau \mapsto \phi(\tau)$ is convex, $\tau\mapsto \Sigma(\vM,\tau)$ is convex as well by Lemma \ref{lemc2} and we get, as in (\ref{beta})
 \be\label{gamma} \dot{\Sigma}(\vM,t)\geq  \dot{\Sigma}(\vM,0) \ . \ee
  where, again, we used that $\bar{A}_1(t)$ is independent of $\gamma$ and $t>0$.
    Recalling $\beta:=1+ t$, let  $\lambda:=\gamma(\beta-1)^2/(2-\beta)$ and $\eps$ small enough we get (\ref{c2}, \ref{barP1}), using (\ref{zirgug},\ref{verified1}, \ref{gamma}).
  \par
   \item{ii)}
    Assume $N=2$, $M_1+M_2=1$, that $\psi_1-\psi_2$ attains its maximum at $x_1$, and $x_2\not=x_1$. Let  $\tilde{\psi}_1:=\beta\psi_1+\lambda \theta_0$ where $\theta_0$ as defined in (\ref{theta0}). We assume, as in part (ii) of the proof of Theorem \ref{new}, that $x_1$ is a maximizer of $\beta\psi_1-\psi_2$ as well.

  Next, assume
 \be\label{this1}\beta\psi_1(x_1)-\psi_2(x_1)< \lambda+\beta\psi_1(x_2)-\psi_2(x_2) \  \ee
 which implies, in particular, that $x_2$ is the maximizer of $\tilde{\psi}_1-\psi_2$ (see part (ii) of the proof of Theorem \ref{new}).
 If, in addition,
 \be\label{contra1}\psi_1(x_1)-\beta\psi_1(x_2)-\lambda-s >0\ ,  \ee
 then,
 from Example \ref{exAlambda},
 \be\label{www1}\bar{P}_1\approx \psi_1(x_1)M_1 >M_1(\tilde{\psi}_1(x_2) +s)  \approx \tilde{P}_1+sM_1
\ee
 if $M_1$ is small, and the proof is obtained.
 \par
 From (\ref{this1}) and since $x_1$ is a maximizer of $\psi_1-\psi_2$:
 $$ \lambda>(\beta-1)(\psi_1(x_1)-\psi_1(x_2))$$
 so  (\ref{contra1}) and (\ref{this1}) are compatible provided
 $$\lambda>(\beta-1)^2\psi_1(x_2)+(\beta-1)\left[\lambda+s\right] \ , $$
 namely
 \be\label{00}\lambda \frac{2-\beta}{\beta-1}>(\beta-1)\psi_1(x_2)+s \ . \ee
  Thus, if we assume further that, say,   $\psi_1(x_2)=0$ (which is consistent with the assumption that $\psi_1, \psi_2\geq 0$) then (\ref{00}) is verified for $s<\lambda(2-\beta)/(\beta-1)$.
 \end{description}
 \end{proof}
\section{Conclusions and further study}\label{differentinter}
We proved the existence of a unique, optimal partitions of a given set of consumers served by a finite set of agents of limited capacities, under certain conditions on their utilities (which were interpreted as the agent's "wisdoms"). This result enabled us to define an individual value of an agent and ask questions about the dependence of this individual value  on her wisdom, where the capacities of all agents and wisdoms of all other agents are  fixed.
\par
Our  definition of individual value can be questioned. In fact, the profit of an agent is just the price she charge per consumer, times the number of her consumers. Since, in our model, the capacity of agent $i$ is given by $M_i$ in the saturation and under saturation case, one may suggest to define her "individual value"
by
$$ V^{a}_i:= M_i p^0_i$$
where $p_i^0$ is the price of her service at equilibrium. Another possible definition is the combined profit of her consumers, which is
$$ V^{c}_i:= \int_{\bar{A}_i}(\psi_i(x)-p^0_i)d\mu$$
where $\bar{A}_i$ is the share of agent $i$ under optimal partition. The individual value as defined in this paper is just the sum
$$\bar{P}_i= V^{a}_i+ V^{c}_i \ . $$
More general definition of individual values can be obtained by different combination of these two values, such as
$$ \bar{P}^{\alpha, \beta}_i:= \alpha V^{a}_i+\beta V^{b}_i \ , $$
for given $\alpha, \beta\geq 0$.
\par
However, the method used in this paper cannot be applied directly  to the case $\alpha\not=\beta$. The reason is that all our results are strongly based on the convexity of   $\Xi:= \Xi(\vt,\vp):=\Xi_{\vt\otimes\vpsi}$    as a function of {\em all} it variables (see (\ref{Xitm})).  In fact, Theorems \ref{new0}-\ref{new1} are all based on the identity $\bar{P}_i:=P_i(\vec{1},\vec{M})$ for the individual value,   where
$$ P_i(\vec{t}, \vec{M})= \frac{\partial \Sigma(\vt, \vM)}{\partial t_i} \   $$
and  $\Sigma(\vt,\vM):= \min_{\vp\in\R^N}\Xi(\vt, \vp)+\vp\cdot\vM$.
\par
In order to extend our results to other definitions of individual value, say $V^a_i$, we my use the duality
$p_i^0=\partial\Sigma/ \partial M_i$ at $\vt=\vec{1}$. Thus $V^a_i:= V_i^a(\vec{1},\vM)$ where
$$  V_i^a(\vec{t},\vM):= M_i\partial_{M_i} \Sigma(\vt, \vM) \ $$
and $V^b_i:= V_i^b(\vec{1}, \vM)$ where
$$ V_i^b(\vt, \vM)=P_i(\vt,\vM)-V^a_i(\vt,\vM)\equiv   \left(\partial_{t_i}-M_i\partial_{M_i}\right)\Sigma
 \ .  $$
The dependence of $V^a_i, V^b_i$ on the wisdom $\psi_i$ for prescribed capacities $\vM$  is related to the behaviour of the functions $t_i\mapsto V_i^a, V_i^b$ at fixed $\vM$.
\par
The difficulty  in analyzing the dependence of either $V^a_i$ or $V^c_i$ on $\psi_i$  is originated from the fact that we do not have any general estimate on the dependence of the partial $M_i$ derivative of $\Sigma$ as functions of $t_i$. All we know about $\Sigma$ is that it is convex in $\vt$ and concave in $\vM$. On the other hand, the dependence of the individual value $\bar{P}_i$, as defined in this paper, is known due to the convexity of $\Sigma$ in the variables $\vt$. Anyway, the analysis of the dependence of $V^a_i, V^b_i$ on $\psi_i$ is worth studying, perhaps by numerical methods.
\vskip .3in\noindent

 Acknowledgment: I would like to thank H. Brezis and R. Holtzmann  for his helpful suggestion which contributed to the improvement of the results.

\section{Appendix}\label{appendix}
We prove that the function $\vm\mapsto\Sigma_{\vpsi}(\vm)$ as defined in (\ref{maximum}) is strictly concave on the simplex (\ref{simplex}).
To prove this we recall some basic elements form convexity theory (see, e.g. \cite{A}, \cite{BC}):
  \begin{description}
  \item{i)} \ If $F$ is a convex function on $\R^N$ (say), then  the  sub gradient $\partial F$ at point $p\in\R^N$ is defined as follows: $\vq\in \partial F(\vp)$ if and only if
 $$ F(\vp\prime)-F(\vp)\geq \vq\cdot(\vp\prime-\vp) \ \ \ \forall \vp\prime\in\R^N \ . $$
\item{ii)} The Legendre transform of $F$:
$$ F^*(\vq):= \sup_{\vp\in\R^N} \vp\cdot\vq-F(\vp) \ ,$$
and $Dom(F^*)\subset\R^N$ is the set on which $F^*<\infty$.
  \item{iii)} The function $F^*$ is convex (and closed), but $Dom(F^*)$ can be a proper subset of $\R^N$ (or even an empty set).
  \item {iv)} The subgradient of a convex function is non-empty (and convex) at any point in the proper domain of this function (i.e. at any point in which the function takes a value in $\R$).
\item{v)} \
Young's  inequality
$$ F(\vp)+F^*(\vq)\geq \vp\cdot\vq$$
holds for any pair of points $(\vp, \vq)\in \R^N\times\R^N$. The equality holds iff $\vq\in\partial F(\vp)$, iff $\vp\in\partial F^*(\vq)$.
\item{vi)} \
The Legendre transform is involuting, i.e $F^{**}=F$ if $F$ is convex and closed.
\end{description}
Returning to our case, let $F(\vp):=\Xi^0_{\vpsi}(\vp)$. It is a convex function on $\R^N$ by Lemma \ref{lemc1}. Moreover, its partial derivatives exists at any point in $\R^N$, which implies that its  sub-gradient is s singleton.   Recalling (\ref{s000}) with $\Xi_{\vpsi}$  we get that
  $$F^*(\vm)\equiv -\Sigma_{\vpsi}(-\vm) $$
  takes finite values  only on the simplex (\ref{simplex}).
 So, we only have to prove the existence of a unique minimizer (\ref{maximum}) on the set $S_N\cap \{\vm\leq \vM\}$, which is a convex set as well. We prove below that $F^*$ is {\em strictly convex} on the simples. This implies the strict concavity of $\Sigma_{\vpsi}$ on the same simplex.

Assume $F^*$ is not strictly convex. It means there exists $\vm_1\not= \vm_2\in S_N$ for which \be\label{just}F^*(\frac{\vm_1+\vm_2}{2})=\frac{F^*(\vm_1)+F^*(\vm_2)}{2} \ . \ee
Let $\vm:=\vm_1/2+\vm_2/2$, and $\vp\in \partial F^*(\vm)$. Then, by (iv, v)
\be\label{just1} 0=F^*(\vm)+ F^{**}(\vp)-\vp\cdot\vm=F^*(\vm) +F(\vp)-\vp\cdot\vm \ . \ee
By (\ref{just}, \ref{just1}):
$$\frac{1}{2}\left( F^*(\vm_1)+F(\vp)-\vp\cdot\vm_1 \right)+ \frac{1}{2}\left( F^*(\vm_2)+F(\vp)-\vp\cdot\vm_2 \right)=0$$
while (v) also guarantees
$$F^*(\vm_i)+F(\vp)-\vp\cdot\vm_i\geq 0 \ \ \ , \ \ i=1,2 \ . $$
It follows
$$F^*(\vm_i)-F(\vp)+\vp\cdot\vm_i= 0 \ \ \ , \ \ i=1,2 \ , $$
so, by (v) again,  $\{\vm_1, \vm_2\}\in \partial F(\vp)$, and we get a contradiction since $\partial F(\vp)$ is a singleton.

\end{document}